\documentclass{amsart}
\usepackage{amssymb,amsmath,amsthm,graphicx,multirow,setspace,url, enumerate, float}
\numberwithin{equation}{section}

\begin{document}

\newtheorem{conjecture}{Conjecture}
\newtheorem{theorem}{Theorem}[section]
\newtheorem{lemma}[theorem]{Lemma}
\newtheorem{corollary}[theorem]{Corollary}

\theoremstyle{definition}
\newtheorem{definition}[theorem]{Definition}
\newtheorem{remark}{Remark}[section]
\newtheorem{example}{Example}[section]

\newtheorem*{theorem*}{Theorem}
\newtheorem*{corollary*}{Corollary}

\title{Odd behavior in the coefficients of reciprocals of binary power series}
\author{Katherine Alexander Anders}\thanks{The author acknowledges support from National Science Foundation grant DMS 08-38434 "EMSW21-MCTP: Research Experience for Graduate Students".}
\address{Department of Mathematics, University of Texas at Tyler, 75799}
\email{kanders@uttyler.edu}
\maketitle

\begin{abstract}
Let $\mathcal{A}$ be a finite subset of $\mathbb{N}$ including $0$ and $f_\mathcal{A}(n)$ be the number of ways to write $n=\sum_{i=0}^{\infty}\epsilon_i2^i$, where $\epsilon_i\in\mathcal{A}$.  The sequence $\left(f_\mathcal{A}(n)\right) \bmod 2$ is always periodic, and $f_\mathcal{A}(n)$ is typically more often even than odd.  We give four families of sets $\left(\mathcal{A}_m\right)$ with $\left|\mathcal{A}_m\right|=4$ such that the proportion of odd $f_{\mathcal{A}_m}(n)$'s goes to $1$ as $m\to\infty$.

\end{abstract}

\section{Introduction}\label{Introduction}
\subsection{Polynomials in $\mathbb{F}_2[x]$}
For a more thorough explanation of the material in this subsection, see Section 3.1 in \cite{Lidl&N}. Note that we are only concerned with polynomials in $\mathbb{F}_2[x]$ rather than  the polynomials over more general finite fields dealt with in \cite{Lidl&N}.

\begin{lemma}[{\cite[1.46]{Lidl&N}}]
For $a,b\in\mathbb{F}_2$ and $n\in\mathbb{N}$, $(a+b)^{2^n}=a^{2^n}+b^{2^n}$.
\end{lemma}

From this lemma and Fermat's Little Theorem, it follows that for any polynomial $f\in\mathbb{F}_2[x]$,
\begin{equation}\label{nice exponents}
f(x)^2=f(x^2) \qquad \text{and thus} \qquad f(x^{2^m})=f(x)^{2^m}. 
\end{equation}

If $f\in\mathbb{F}_2[x]$ with $f(0)\neq 0$ and $\operatorname{deg}(f)=n\geq 1$, then there exists $D\in\mathbb{Z}$ with $1\leq D\leq 2^n-1$ such that $f(x)\mid 1+x^D$.  The least such $D$ is called the \textit{order} of $f$ and is denoted $\operatorname{ord}(f(x))=\operatorname{ord}(f)$.  If $f\in\mathbb{F}_2[x]$ is an irreducible polynomial over $\mathbb{F}_2$ with $\operatorname{deg}(f)=n$, then $\operatorname{ord}(f)\mid2^n-1$.

\begin{theorem}[{\cite[3.6]{Lidl&N}}]\label{order divides}
Let $c\in\mathbb{Z}$ with $c>0$ and $f\in\mathbb{F}_2[x]$ with $f(0)\neq0$. Then $f(x)\mid 1+x^c$ if and only if $\operatorname{ord}(f)\mid c$.
\end{theorem}

\begin{definition}\label{reciprocal polynomial}
For a polynomial $f(x)$ of degree $n$, the \textit{reciprocal polynomial} of $f(x)$ is $f_{(R)}(x):=x^nf(1/x)$.
\end{definition}

\begin{theorem}[{\cite[3.13]{Lidl&N}}]\label{reciprocal matches order}
If $\operatorname{ord}(f(x))=D$, then $\operatorname{ord}\left(f_{(R)}(x)\right)=D$.
\end{theorem}

\begin{definition}\label{ell_1 definition}
For a polynomial $f(x)\in\mathbb{F}_2[x]$, define the \textit{length} of $f(x)$ to be the number of monomials in $f(x)$.  This can also be viewed as the number of terms in $f(x)$ with coefficient $1$ and is denoted by $\ell_1(f(x))$.
\end{definition}

\begin{definition}\label{ell_0 definition}
For a polynomial $f(x)\in\mathbb{F}_2[x]$, let $\ell_{0,N}(f(x))$ denote the number of terms in $f(x)$ with coefficient $0$ when $f(x)$ is viewed as a polynomial of degree $N$.  Note that $N$ may exceed $d$, the usual degree of $f(x)$, if we take all terms of the form $x^k$, where $k>d$, to have coefficient $0$.
\end{definition}

Let $f(x)\in\mathbb{F}_2[x]$ with $\operatorname{ord}(f(x))=D$, $f(0)=1$, and $f(x)f^*(x)=1+x^D$.  In \cite{KOB paper}, Cooper, Eichhorn, and O'Bryant considered fractions of the form 
\[
\gamma(f(x)):=\frac{\ell_1(f^*(x))}{D}
\]
for various polynomials $f(x)$, as did we in \cite{4 paper}.  Here we instead consider the ordered pair
\begin{equation}\label{first beta definition}
\beta(f(x)):=(\ell_1(f^*), \ell_{0,D-1}(f^*)),
\end{equation}
which gives more precise information than reduced fractions would.

\begin{definition}
We call a polynomial $f(x)$ \textit{robust} if the first coordinate of $\beta(f(x))$ exceeds the second coordinate by more than one, so $\ell_1(f^*(x))>\ell_{0,D-1}(f^*(x))+1$, where $D$ is the order of $f(x)$.  This is equivalent to saying that $\ell_1(f^*(x))>(D+1)/2$.
\end{definition}

\begin{remark}\label{2/3 remark}
Suppose $f(x)$ is not robust. If $\beta(f(x))=(1,0)$, then $\gamma(f(x))=1$.  Otherwise, $\beta(f(x))$ is of the form $(a,b)$, where $b\geq1$ and $a\leq b+1$.  Let $\theta(x)=\frac{x}{x+1}$.  Note that $\theta(x)$ is increasing for $x\geq 0$ and
\[
 \frac{a}{a+b}=\frac{\frac{a}{b}}{\frac{a}{b}+1}=\theta\left(\frac{a}{b}\right).
\]
Since $\frac{a}{b}\leq \frac{b+1}{b}=1+\frac{1}{b}\leq 2$, it follows that
\[
 \gamma(f(x))=\frac{a}{a+b}=\theta\left(\frac{a}{b}\right)\leq \theta(2)=\frac{2}{3}.
\]
Hence if $\gamma(f(x))>2/3$, then $f(x)$ is robust.
\end{remark}

In \cite{KOB paper}, Cooper, Eichhorn, and O'Bryant posed the open question of describing the set
\[
\mathcal{U}:=\{\gamma(f(x)): f \text{ is a polynomial}\}.
\]
They showed that of the $2048$ polynomials of degree less than $12$ with constant term $1$, $421$ have $\gamma(f(x))=1/2$.  They also showed by direct computation that no polynomial in $\mathbb{F}_2[x]$ of degree less than $8$ is robust.  Thus for all polynomials of degree less than $8$,$\gamma(f(x))<(\operatorname{ord}(f)+1)/2$.  In fact, if $\operatorname{deg}(f(x))<8$ and $\gamma(f(x))>1/2$, then $\gamma(f(x))=m/(2m-1)$ for some $m\in\mathbb{N}$.

Using the notation of \cite{KOB paper}, for a given positive integer $n$, let $P_n$ denote the polynomial in $\mathbb{F}_2[x]$ whose exponents are the powers of 2 in the binary representation of $n$.  This enumerates $\mathbb{F}_2[x]$.  For example, $11=[1011]_2$, so $P_{11}(x)=x^3+x+1$.

Figure \ref{PolynomialReciprocalDensity} was taken from \cite{KOB paper} with permission and gives a dot plot of all points of the form $\left(n,\gamma\left(P_n\right)\right)$ for $n$ odd and less than $2^{12}$.  The points are tightly clustered around $1/2$, but when they stray from $1/2$, there is a strong propensity to be smaller than $1/2$ rather than greater.  Note the four points near the top represented by boxes.  We will explain and generalize these robust polynomials in Section \ref{Families of Robust Polynomials}.

\begin{figure}[h!]
	\begin{center}
		\includegraphics[width=0.75\textwidth]{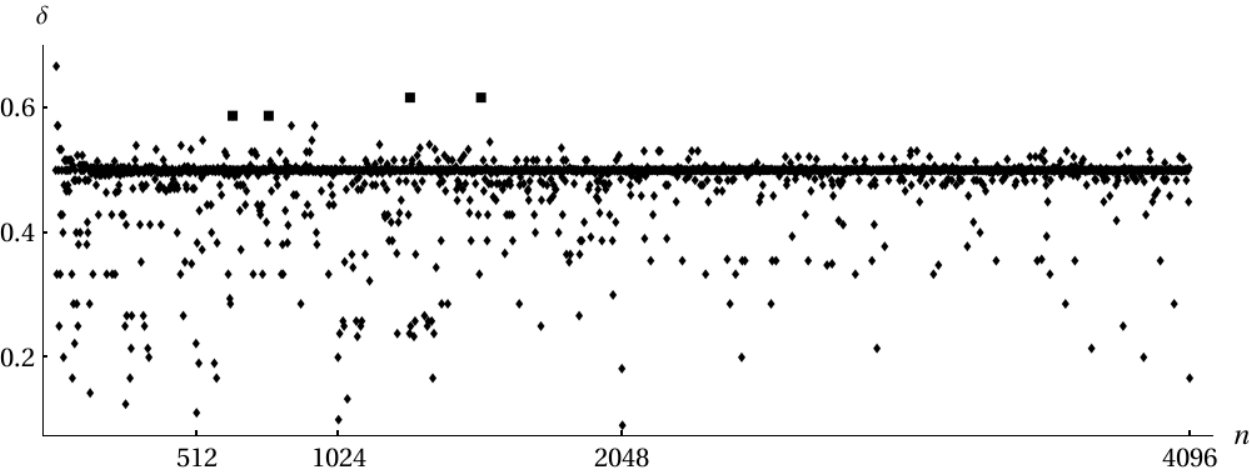}
	\end{center}
\caption{The points $\left(n,\gamma\left(P_n\right)\right)$ with $n$ odd, except $(1,0)$ and $(3,1)$, taken from \cite{KOB paper}}\label{PolynomialReciprocalDensity}
\end{figure}

Since $1+x^D$ has order $D$ and $\ell_1\left(\left(1+x^D\right)^*\right)=\ell_1(1)=1$, we see $\gamma\left(1+x^D\right)=1/D$ and the lower bound of $\mathcal{U}$ is $0$.

The main results of this paper establish $1$ as the supremum of $\mathcal{U}$.  They will be shown in Section \ref{Families of Robust Polynomials} and are summarized below.  Section \ref{Applications} discusses applications of the main results to generalized binary representations, and Section \ref{Open Questions}  is a discussion of open questions on these problems and areas for future work.

\begin{theorem*}
The polynomials $f_{r,1}(x)=1+x+x^{2^r-1}+x^{2^r+1}$ are robust with order dividing $4^r-1$ and $\displaystyle\lim_{r\to\infty} \gamma\left(f_{r,1}\right)=1$.
\end{theorem*}

\begin{corollary*}
The reciprocal polynomials $f_{(R),r,1}(x)=1+x^2+x^{2^r}+x^{2^r+1}$ are robust with order dividing $4^r-1$ and $\displaystyle\lim_{r\to\infty} \gamma\left(f_{(R),r,1}\right)=1$.
\end{corollary*}

\begin{theorem*}
The polynomials $f_{r,2}(x)=1+x+x^{2^r}+x^{2^r+2}$ are robust with order dividing $4^r+2^r+1$ and $\displaystyle\lim_{r\to\infty} \gamma\left(f_{r,2}\right)=1$.
\end{theorem*}

\begin{corollary*}
The reciprocal polynomials $f_{(R),r,2}(x)=1+x^2+x^{2^r+1}+x^{2^r+2}$ are robust with order dividing $4^r+2^r+1$ and $\displaystyle\lim_{r\to\infty} \gamma\left(f_{(R),r,2}\right)=1$.
\end{corollary*}

\subsection{Acknowledgements}
The results in this paper were part of the author's doctoral dissertation at the University of Illinois at Urbana-Champaign.  The author expresses appreciation to Joshua Cooper, Dennis Eichhorn, and Kevin O'Bryant for permission to adapt Figure \ref{PolynomialReciprocalDensity} from their paper \textit{Reciprocals of binary power series}, \cite{KOB paper}.  The author also wishes to thank Professor Bruce Reznick for his time, ideas, and encouragement.

\section{Families of Robust Polynomials}\label{Families of Robust Polynomials}
In this section, we present four sequences $\{f_n\}$ of polynomials such that
\[
\lim_{n\to\infty}\gamma\left(f_n\right)=1, 
\]
thereby establishing $1$ as the least upper bound of the set $\mathcal{U}$.  We then consider examples of elements of these sequences.  At the end of the section, we discuss the methods of data collection used in finding these and other examples of robust polynomials.  All polynomials in this subsection are considered as elements of $\mathbb{F}_2[x]$.  

Equation \eqref{first beta definition} defined $\beta(f(x))$, and we now define the more general ordered pair $\beta_N(f(x))$.

\begin{definition}
For $f(x)\in\mathbb{F}_2[x]$ and $N$ a multiple of the order of $f(x)$ with $f^*(x):=(1+x^N)/f(x)$, we define
\[
\beta_N(f(x))=\left(\ell_1(f^*(x), \ell_{0,N-1}(f^*(x))\right).
\]
\end{definition}

\begin{lemma}\label{robustness of reciprocals}
In $\mathbb{F}_2[x]$, $\beta(f(x))=\beta\left(f_R(x)\right)$, and the robustness of $f(x)$ is equivalent to the robustness of $f_R(x)$.
\end{lemma}

\begin{proof}
According to Theorem \ref{reciprocal matches order}, if $\operatorname{order}(f(x))=D$, then $\operatorname{order}\left(f_R(x)\right)=D$.  Suppose $\operatorname{deg}(f(x))=n$.  Then we have
\[
f(x)f^*(x)=1+x^D
\]
and
\[
f_{(R)}(x)\left(f_{(R)}\right)^*(x)=1+x^D,
\]
where $\operatorname{deg}(f^*(x))=\operatorname{deg}\left(\left(f_{(R)}\right)^*(x)\right)=D-n$.

Now 
\begin{align*}
\left(f^*\right)_{(R)}(x)&=x^{D-n}f^*\left(\frac{1}{x}\right)=x^{D-n}\left(\frac{1+\left(\frac{1}{x}\right)^D}{f\left(\frac{1}{x}\right)}\right)\\
&=\frac{x^D\left(1+\frac{1}{x^D}\right)}{x^nf\left(\frac{1}{x}\right)}=\frac{1+x^D}{f_{(R)}(x)}\\
&=\left(f_{(R)}\right)^*(x).
\end{align*}
Thus there is no ambiguity in writing $f_{(R)}^*(x)$, and $\ell_1\left(f^*(x)\right)=\ell_1\left(f_{(R)}^*(x)\right)$, so we see that $\beta(f(x))=\beta\left(f_R(x)\right)$.
\end{proof}

\begin{lemma}\label{exact order unnecessary}
If $f(x),g(x),h(x)\in\mathbb{F}_2[x]$ satisfy $f(x)g(x)=1+x^N$ and $f(x)h(x)=1+x^M$, where $N<M$, then $\ell_1(g(x))/N=\ell_1(h(x))/M$.  In particular, if $\ell_1(g(x))/N$ is in lowest terms, then $N$ is the order of $f(x)$.
\end{lemma}

\begin{proof}
From Theorem \ref{order divides}, we know that $N\mid M$, so $M=jN$ .  We can write
\[
h(x)=\frac{1+x^M}{f(x)}=g(x)\cdot\frac{1+x^{jN}}{1+x^N}=g(x)(1+x^N+\cdots+x^{(j-1)N}).
\]
If $\ell_1(g(x))=k$ and $\ell_{0,N-1}(g(x))=N-k$, then $\ell_1(h(x))=kM/N=jk$ and $\ell_{0,M-1}(h(x))=(N-k)M/N=M-jk=j(N-k)$.  This proves the assertion.
\end{proof}

\begin{definition}
For a non-negative integer $k$, let $b(k)$ denote the number of $1$'s in the standard binary representation of $k$.
\end{definition}

\begin{lemma}\label{Glaisher's Theorem}
For $r\geq2$,
\[
 \sum_{k=0}^{2^r-2} 2^{b(k)}=3^r-2^r.
\]

\end{lemma}

\begin{proof}
Since $b(2^r-1)=r$ with $r$ digits in the representation and no zeros, if $0\leq n\leq 2^r-2$, then $b(n)\leq r-1$.  Consider counting the value of $\sum_{k=0}^{2^r-2} 2^{b(k)}$ by first fixing the value of $b(k)$.  Let $b(k)=j$, where $0\leq j\leq r-1$.  There are $\binom{r}{j}$ numbers $n$ in the range of summation with $b(n)=j$.  Hence the contribution to the sum from numbers with $b(k)=j$ is $\binom{r}{j}2^j$.  Using this and the binomial formula, we obtain
\begin{align*}
 \sum_{k=0}^{2^r-2} 2^{b(k)}&=\binom{r}{0}2^0+\binom{r}{1}2^1+\binom{r}{2}2^2+\cdots+\binom{r}{r-1}2^{r-1}\\
&=(2+1)^r-2^r\\
&=3^r-2^r.
\qedhere
\end{align*}
\end{proof}

\begin{lemma}\label{a,b lemma}
For $a,b\in\mathbb{N}$, 
\[
(1+x^a+x^b)\prod_{j=0}^{m-1}\left(1+x^{2^ja}+x^{2^jb}\right)=1+x^{2^ma}+x^{2^mb}.
\]
\end{lemma}

\begin{proof}
Let $m=1$. Then the product is $(1+x^a+x^b)(1+x^a+x^b)=(1+x^{2a}+x^{2b})$ by Equation \eqref{nice exponents}.  Suppose the result holds for all $1\leq m\leq n$.
Then
\begin{align*}
 (1+x^a+x^b)\prod_{j=0}^n\left(1+x^{2^ja}+x^{2^jb}\right)&=(1+x^{2^na}+x^{2^nb})(1+x^{2^na}+x^{2^nb})\\
&=\left(1+(x^2)^{2^na}+(x^2)^{2^nb}\right)\\
&=1+x^{2^{n+1}a}+x^{2^{n+1}b},
\end{align*}
where we have again used Equation \eqref{nice exponents}.
Hence by induction the result holds for all $m$.
\end{proof}

\begin{lemma}\label{in general lemma}
For $1\leq r\in\mathbb{N}$,
\[
\left(1+x^{2^r-1}+x^{2^r}\right)\left(\prod_{j=0}^{r-1}\left(1+x^{(2^r-1)2^j}+x^{2^r2^j}\right)+x^{4^r-2^r}\right)=1+x^{4^r-1}.
\]
\end{lemma}

\begin{proof}
Using Lemma \ref{a,b lemma} with $a=2^r-1, b=2^r$, and $m=r$, 
\begin{align*}
\left(1+x^{2^r-1}+x^{2^r}\right)&\left(\prod_{j=0}^{r-1}\left(1+x^{(2^r-1)2^j}+x^{2^r2^j}\right)+x^{4^r-2^r}\right)\\
&=1+x^{2^r(2^r-1)}+x^{2^r(2^r)}+\left(1+x^{2^r-1}+x^{2^r}\right)x^{4^r-2^r}\\
&=1+x^{4^r-2^r}+x^{4^r}+x^{4^r-2^r}+x^{4^r-1}+x^{4^r}\\
&=1+x^{4^r-1}.
\qedhere
\end{align*}
\end{proof}

Let $d_{r,1}=3^r-1$ and $c_{r,1}=(4^r-1)-d_{r,1}=4^r-3^r$.

\begin{theorem}\label{maintheorem}
Fix $r\geq3$. 
\begin{enumerate}[\upshape (i)]
\item The order of $f_{r,1}(x):=1+x+x^{2^r-1}+x^{2^r+1}$ divides $4^r-1$.\label{order 1}
\item The polynomial $h_{r,1}(x):=(1+x^{4^r-1})/f_{r,1}(x)=f_{r,1}^*$ has $\ell_1(h_{r,1}(x))=c_{r,1}$.\label{length 1}
\item Hence $\beta_{4^r-1}(f_{r,1})=(c_{r,1},d_{r,1})$ and $f_{r,1}(x)$ is robust.\label{robust 1}
\end{enumerate}
\end{theorem}

\begin{proof}
Define 
\begin{equation}\label{product representation for g_r,1}
g_{r,1}(x)=\prod_{j=0}^{r-1}\left(1+x^{(2^r-1)2^j}+x^{2^r2^j}\right)+x^{4^r-2^r}.
\end{equation}
Then Lemma \ref{in general lemma} gives
\[
\left(1+x^{2^r-1}+x^{2^r}\right)g_{r,1}(x)=1+x^{4^r-1}.
\]
Since 
\[
g_{r,1}(1)=\prod_{j=0}^{r-1}\left(1+1+1\right) +1\equiv0\pmod2,
\]
we know $(1+x)\mid g_{r,1}(x)$.  Hence there exists $h_{r,1}(x)\in\mathbb{F}_2[x]$ such that $(1+x)h_{r,1}(x)=g_{r,1}(x)$, so
\[
\left(1+x^{2^r-1}+x^{2^r}\right)(1+x)h_{r,1}(x)=1+x^{4^r-1}.
\]
Since $f_{r,1}(x)=1+x+x^{2^r-1}+x^{2^r+1}=(1+x)\left(1+x^{2^r-1}+x^{2^r}\right)$, we see that $f_{r,1}(x)\mid \left(1+x^{4^r-1}\right)$.  We have not shown that $4^r-1$ is actually the order of $f_{r,1}(x)$, but we know by Lemma \ref{exact order unnecessary} that the exact order is not necessary to determine robustness.  We have checked by direct computation that for $r\leq 10$, $4^r-1$ is the exact order of $f_{r,1}$.

Now we seek a nice expression for $h_{r,1}(x)$ to use in proving part (\ref{length 1}).  We will do this by manipulating $g_{r,1}(x)$.  Rewrite \eqref{product representation for g_r,1} to obtain
\begin{equation}\label{factored product representation for g_r,1}
g_{r,1}(x)=\prod_{j=0}^{r-1}\left(1+x^{(2^r-1)2^j}(1+x^{2^j})\right)+x^{4^r-2^r}.
\end{equation}
We next expand the product in \eqref{factored product representation for g_r,1} and use Equation \eqref{nice exponents}, specifically $1+x^{2^j}=(1+x)^{2^j}$, to see that, with the exception of $1$ and $x^{4^r-2^r}$, all summands in the expanded product are terms of the form $x^{(2^r-1)\sum2^i}(1+x)^{\sum2^i}$, where $\sum2^i$ is a sum of some subset of $\{2^0,2^1,\ldots,2^{r-1}\}$.  Considering all such $\sum2^i$, we get all terms of the form $x^{(2^r-1)n}(1+x)^n$ for $1\leq n\leq 2^r-1$.  Thus we can rewrite \eqref{factored product representation for g_r,1} as
\begin{align*}
g_{r,1}(x)&=1+x^{4^r-2^r}+\sum_{n=1}^{2^r-1} x^{(2^r-1)n}(1+x)^n\\ &=(1+x)\left(\frac{1+x^{4^r-2^r}}{1+x}+\sum_{n=1}^{2^r-1} x^{(2^r-1)n}(1+x)^{n-1}\right)\\
&=(1+x)\left(\sum_{j=0}^{4^r-2^r-1} x^j +\sum_{n=1}^{2^r-1} x^{(2^r-1)n}(1+x)^{n-1}\right).
\end{align*}
Hence by the definition of $h_{r,1}(x)$,
\begin{equation*}
h_{r,1}(x)=\sum_{j=0}^{4^r-2^r-1} x^j+\sum_{n=1}^{2^r-1} x^{(2^r-1)n}(1+x)^{n-1}.
\end{equation*}

We shall use this representation of $h_{r,1}(x)$ to determine $\ell_1(h_{r,1}(x))$.  We begin by focusing on 
\[
S_{r,1}(x):=\sum_{n=1}^{2^r-1} x^{(2^r-1)n}(1+x)^{n-1},
\]
which is a polynomial of degree $4^r-2^r-1$.  We note that the greatest exponent in a monomial when $n=k$ is $(2^r-1)k+(k-1)=2^rk-1$, and the least exponent in a monomial when $n=k+1$ is $(2^r-1)(k+1)=2^rk+2^r-(k+1)$.  Since $k+1\leq2^r-1$, it follows that $2^rk-1<2^rk+2^r-(k+1)$, so there is no cancellation of terms within $S_{r,1}(x)$.  Glaisher's Theorem, see \cite{Beeblebrox}, states that the number of odd binomial coefficients of the form $\binom{n}{j}$, $0\leq j\leq n$, is equal to $2^{b(j)}$.  Using this and Lemma \ref{Glaisher's Theorem}, we see that
\[
 \ell_1(S_{r,1}(x))=\sum_{j=1}^{2^r-1} 2^{b(j-1)}=\sum_{k=0}^{2^r-2} 2^{b(k)}=3^r-2^r.
\]
Since $S_{r,1}(x)$ is a polynomial of degree $4^r-2^r-1$, $S_{r,1}(x)$ has $4^r-2^r$ possible terms and $\ell_{0,4^r-2^r-1}\left(S_{r,1}(x)\right)=4^r-2^r-(3^r-2^r)=4^r-3^r$.  Then, to construct $h_{r,1}(x)$, we add $\sum_{j=0}^{4^r-2^r-1} x^j$.  Note that the degree of this sum is equal to the degree of $S_{r,1}(x)$.  This addition has the effect of reversing the $0$'s and $1$'s, so $\ell_1(h_{r,1}(x))=4^r-3^r$ and $\ell_{0,4^r-2^r-1}(h_{r,1}(x))=3^r-2^r$, completing the proof of part (\ref{length 1}).  Because the order of $f_{r,1}(x)$ divides $4^r-1$, we consider $h_{r,1}(x)$ as a polynomial of degree $4^r-2$ with $4^r-1$ possible terms.  The $2^r-1$ terms of degree $4^r-2^r$,\dots,$4^r-2$ have coefficient $0$.  Thus in total $\ell_1(h_{r,1}(x))=4^r-3^r=c_{r,1}$ and $\ell_{0,4^r-2}(h_{r,1}(x))=3^r-1=d_{r,1}$.  

Since $\operatorname{gcd}(c_{3,1},d_{3,1})=\operatorname{gcd}(37,26)=1$, we know $\ell_1\left(h_{3,1}\right)/(4^3-1)$ is in lowest terms.  By Lemma \ref{exact order unnecessary}, the order of $f_{3,1}$ is indeed $4^3-1$, and the polynomial is robust.  For $r\geq4$, it is not necessarily the case that $\operatorname{gcd}(c_{r,1},d_{r,1})=1$, but it is true that
\begin{align*}
\frac{c_{r,1}}{4^r-1}&=\frac{4^r-3^r}{4^r-1}=1-\frac{3^r-1}{4^r-1}\\
&>1-\frac{3^r-(3/4)^r}{4^r-1}=1-\left(\frac{3}{4}\right)^r>\frac{2}{3},
\end{align*}
so $f_{r,1}(x)$ is robust by Remark \ref{2/3 remark}.
\end{proof}

\begin{example}\label{f_3,1}
Consider $f_{3,1}(x)=1+x+x^7+x^9$.  The order of $f_{3,1}(x)$ is $4^3-1=63$.  The polynomial $f_{3,1}^*(x)$ has $\ell_1\left(f_{3,1}^*(x)\right)=4^3-3^3=37$, and $\beta\left(f_{3,1}(x)\right)=(37,26)$.  Explicitly,
\begin{align*}
 f_{3,1}^*(x)=&x^{54}+x^{52}+x^{50}+x^{48}+x^{45}+x^{44}+x^{41}+x^{40}+x^{38}+x^{37}+x^{36}+
   x^{34}\\
  &+x^{33}+x^{32}+x^{27}+x^{26}+x^{25}+x^{24}+x^{22}+x^{20}+x^{19}+x^{1
   8}+x^{17}\\
  &+x^{16}+x^{13}+x^{12}+x^{11}+x^{10}+x^9+x^8+x^6+x^5+x^4+x^3+x^2+x
   +1.
\end{align*}

\end{example}

\begin{corollary}\label{maintheorem corollary}
The reciprocal polynomials $f_{(R),r,1}=1+x^2+x^{2^r}+x^{2^r+1}$ are robust with order dividing $4^r-1$.
\end{corollary}

\begin{proof}
This follows immediately from Theorem \ref{maintheorem} and Lemma \ref{robustness of reciprocals}.
\end{proof}

\begin{example}\label{f_(R),3,1}
Consider $f_{(R),3,1}(x)=1+x^2+x^8+x^9$.  The order of $f_{(R),3,1}(x)$ is $4^3-1=63$.  The polynomial $f_{(R),3,1}^*(x)$ has $\ell_1\left(f_{(R),3,1}^*(x)\right)=4^3-3^3=37$, and $\beta\left(f_{(R),3,1}(x)\right)=(37,26)$.
\end{example}

We now exhibit another family of robust polynomials.

Let $c_{r,2}=4^r-3^r+2^r$ and $d_{r,2}=3^r+1$.

\begin{theorem}\label{Second Main Theorem}
Fix $r\geq3$.  
\begin{enumerate}[\upshape (i)]
\item The order of $f_{r,2}(x):=1+x+x^{2^r}+x^{2^r+2}$ divides $4^r+2^r+1$.\label{order 2} 
\item The polynomial $h_{r,2}(x):=(1+x^{4^r+2^r+1})/f_{r,2}(x)=f_{r,2}^*$ has $\ell_1(h_{r,2}(x))=c_{r,2}$.\label{length 2}  
\item Hence $\beta_{4^r+2^r+1}(f_{r,2}(x))=(c_{r,2},d_{r,2})$ and $f_{r,2}(x)$ is robust.\label{robust 2}
\end{enumerate}
\end{theorem}

\begin{proof}
Let
\begin{equation}\label{product representation of g_r,2}
g_{r,2}(x)=\prod_{j=0}^{r-1} \left(1+x^{2^j2^r}+x^{2^j(2^r+1)}\right).
\end{equation}
By Lemma \ref{a,b lemma}, we know that
\begin{align}\label{a,b lemma applied to g_r,2}
\left(1+x^{2^r}+x^{2^r+1}\right)g_{r,2}(x)&=1+x^{2^r2^r}+x^{2^r(2^r+1)}\\
&=1+x^{4^r}+x^{4^r+2^r}\notag.
\end{align}
By factoring the terms in \eqref{product representation of g_r,2} we obtain
\begin{align*}
g_{r,2}(x)&=\prod_{j=0}^{r-1}\left(1+x^{2^j2^r}\left(1+x^{2^j}\right)\right).
\end{align*}
Then by expanding the product and using Equation \eqref{nice exponents}, we see that, with the exception of the term $1$, all summands in the expanded product are terms of the form $x^{2^r\sum2^i}(1+x)^{\sum2^i}$, where $\sum2^i$ is a sum of some subset of $\{2^0,2^1,\ldots,2^{r-1}\}$.  Considering all such $\sum2^i$, we get all terms of the form $x^{2^rn}(1+x)^n$ for all $1\leq n\leq 2^r-1$.  Thus we can rewrite \eqref{product representation of g_r,2} as
\begin{align}\label{g_r,2 factorization and sum}
g_{r,2}(x)&=\prod_{j=0}^{r-1}\left(1+x^{2^j2^r}\left(1+x^{2^j}\right)\right)\\
&=1+\sum_{i=1}^{2^r-1}x^{2^ri}(1+x)^i\notag.
\end{align}
Using equations \eqref{a,b lemma applied to g_r,2} and \eqref{g_r,2 factorization and sum}, we see that
\begin{align*}
\left(1+x^{2^r}+x^{2^r+1}\right)&\left(1+x^{4^r}+\sum_{i=1}^{2^r-1}x^{2^ri}(1+x)^i\right)\\
&=\left(1+x^{2^r}+x^{2^r+1}\right)\left(x^{4^r}+g_{r,2}(x)\right)\\
&=\left(1+x^{2^r}+x^{2^r+1}\right)x^{4^r}+1+x^{4^r}+x^{4^r+2^r}\\
&=x^{4^r}+x^{4^r+2^r}+x^{4^r+2^r+1}+1+x^{4^r}+x^{4^r+2^r}\\
&=1+x^{4^r+2^r+1}.
\end{align*}
Now observe that
\begin{align*}
\left(1+x^{2^r}+x^{2^r+1}\right)&\Big(1+x\Big)\left(\frac{1+x^{4^r}}{1+x}+\sum_{i=1}^{2^r-1}x^{2^ri}(1+x)^{i-1}\right)\\
&=\left(1+x+x^{2^r}+x^{2^r+2}\right)\left(\frac{1+x^{4^r}}{1+x}+\sum_{i=1}^{2^r-1}x^{2^ri}(1+x)^{i-1}\right)\\
&=f_{r,2}(x)\left(\frac{1+x^{4^r}}{1+x}+\sum_{i=1}^{2^r-1}x^{2^ri}(1+x)^{i-1}\right)\\
&=1+x^{4^r+2^r+1}.
\end{align*}
Thus the order of $f_{r,2}(x)$ divides $4^r+2^r+1$, completing the proof of part (\ref{order 2}), and that suffices to determine if $f_{r,2}(x)$ is robust by Lemma \ref{exact order unnecessary}.  We have checked by direct computation that $4^r+2^r+1$ is the exact order of $f_{r,2}(x)$ when $r\leq 10$.

Let 
\[
S_{r,2}(x):=\sum_{i=1}^{2^r-1}x^{2^ri}(1+x)^{i-1},
\]
so $h_{r,2}(x)=\frac{1+x^{4^r}}{1+x}+S_{r,2}(x)$.  We wish to determine $\ell_1(h_{r,2}(x))$ and will begin by determining $\ell_1(S_{r,2}(x))$.  We first note that when $i=k$, the monomial of greatest degree is $x^{2^rk}x^{k-1}=x^{2^rk+k-1}$.  When $i=k+1$, the monomial of lowest degree is $x^{2^r(k+1)}=x^{2^rk+2^r}$.  Since $k<2^r-1$, it follows that $2^rk+k-1<2^rk+2^r$, so there is no overlap of terms from $i=k$ and $i=k+1$.

Once again, we use Glaisher's Theorem, see \cite{Beeblebrox}, which states that the number of odd binomial coefficients of the form $\binom{n}{j}$, $0\leq j\leq n$, is equal to $2^{b(j)}$, and Lemma \ref{Glaisher's Theorem} to see that
\[
 \ell_1(S_{r,2}(x))=\sum_{j=1}^{2^r-1} 2^{b(j-1)}=\sum_{k=0}^{2^r-2} 2^{b(k)}=3^r-2^r.
\]

Because $S_{r,2}(x)$ is a polynomial of degree $2^r(2^r-1)+2^r-2=4^r-2$, we have $\ell_{0,4^r-2}(S_{r,2}(x))=4^r-2+1-3^r+2^r=4^r-3^r+2^r-1$.  Adding in the $(1+x^{4^r})/(1+x)=1+x+x^2+\cdots+x^{4^r-2}+x^{4^r-1}$ to construct $h_{r,2}(x)$ has the effect of reversing the $0$'s and $1$'s and adding an additional $1$.  Hence $\ell_1(h_{r,2}(x))=4^r-3^r+2^r$ and $\ell_{0,4^r-2}(h_{r,2}(x))=3^r-2^r$, and the proof of part (\ref{length 2}) is complete.  We now consider $h_{r,2}(x)$ as a polynomial of degree $4^r+2^r$, so the remaining $4^r+2^r-4^r+1=2^r+1$ terms have coefficient $0$.  Hence $\ell_1(h_{r,2}(x))=4^r-3^r+2^r=c_{r,2}$ and $\ell_{0,4^r+2^r}(h_{r,2}(x))=3^r+1=d_{r,2}$.  

It is not necessarily the case that $\operatorname{gcd}(c_{r,2},d_{r,2})=1$, and when this fails we know only that the order of $f_{r,2}(x)$ divides $4^r+2^r+1$, but this is still sufficient to determine if $f_{r,2}(x)$ is robust by Lemma \ref{exact order unnecessary}.  In fact, $\operatorname{gcd}(c_{6,2},d_{6,2})\neq1$, but $4^r+2^r+1$ is indeed the order of $f_{6,2}(x)$ and not just a divisor of the order.  For $1\leq r\leq 5$, $\operatorname{gcd}(c_{r,2},d_{r,2})=1$, so the order of $f_{r,2}(x)$ is $4^r+2^r+1$, and $\beta(f_{r,2}(x))=(c_{r,2},d_{r,2})$, making $f_{r,2}(x)$ robust.  For $r\geq 6$,
\begin{align*}
\frac{c_{r,2}}{4^r+2^r+1}&=\frac{4^r-3^r+2^r}{4^r+2^r+1}
=1-\frac{3^r+1}{4^r+2^r+1}\\
&>1-\frac{3^r+(3/2)^r+(3/4)^r}{4^r+2^r+1}
=1-\left(\frac{3}{4}\right)^r
>\frac{2}{3},
\end{align*}
so $f_{r,2}(x)$ is robust by Remark \ref{2/3 remark}.
\end{proof}

\begin{example}\label{f_3,2}
Consider $f_{3,2}(x)=1+x+x^8+x^{10}$.  The order of $f_{3,2}(x)$ is $4^3+2^3+1=73$.  The polynomial $f_{3,2}^*(x)$ has $\ell_1\left(f_{3,2}^*(x)\right)=4^3-3^3+2^3=45$, and $\beta\left(f_{3,2}(x)\right)=(45,28)$.  Explicitly,
\begin{align*}
f_{3,2}^*(x)=&x^{63}+x^{61}+x^{59}+x^{57}+x^{55}+x^{54}+x^{51}+x^{50}+x^{47}+x^{46}+x^{45}+
   x^{43}\\
  &+x^{42}+x^{41}+x^{39}+x^{38}+x^{37}+x^{36}+x^{31}+x^{30}+x^{29}+x^{2
   8}+x^{27}+x^{25}\\
  &+x^{23}+x^{22}+x^{21}+x^{20}+x^{19}+x^{18}+x^{15}+x^{14}+x
   ^{13}+x^{12}+x^{11}\\
  &+x^{10}+x^9+x^7+x^6+x^5+x^4+x^3+x^2+x+1.
\end{align*}
\end{example}

\begin{corollary}\label{Second Main Theorem Corollary}
The reciprocal polynomials $f_{(R),r,2}(x)=1+x^2+x^{2^r+1}+x^{2^r+2}$ are robust with order dividing $4^r+2^r+1$.
\end{corollary}

\begin{proof}
This follows immediately from Theorem \ref{Second Main Theorem} and Lemma \ref{robustness of reciprocals}.
\end{proof}

\begin{example}\label{f_(R),3,2}
Consider $f_{(R),3,2}(x)=1+x^2+x^9+x^{10}$.  The order of $f_{(R),3,2}(x)=4^3+2^3+1=73$.  The polynomial $f_{(R),3,2}^*(x)$ has $\ell_1\left(f_{(R),3,2}^*(x)\right)=4^3-3^3+2^3=45$, and $\beta\left(f_{(R),3,2}\right)=(45,28)$.
\end{example}

With Examples \ref{f_3,1}, \ref{f_(R),3,1}, \ref{f_3,2}, and \ref{f_(R),3,2}, we have accounted for all of the rectangular points in Figure \ref{PolynomialReciprocalDensity}.

In our search for robust polynomials, we have used Mathematica to check all polynomials of order less than or equal to 83, all quadrinomials of degree less than or equal to 18, all trinomials of degree less than or equal to 19, and all polynomials of degree less than or equal to 14.  

Table V-1 of \cite{Golomb} contains information, including orders, on trinomials of degree less than or equal to $36$.  We used this information on orders to obtain $\beta(f(x))$ for all trinomials $f(x)\in\mathbb{F}_2[x]$ with degree less than or equal to $19$.  There were only $4$ robust trinomials in this range, and they are given in Table \ref{robust trinomials table}.  Calculations became difficult for trinomials of higher degree because of the large amounts of time needed to run the code.

\begin{table}[h!]
\begin{center}
\begin{tabular}{l|l| l}
$f(x)$ &$\operatorname{ord}(f(x))$ &$\beta(f(x))$\\ \hline
&&\\
$1+x^3+x^{14}$ &$5115$ &$(2600,2515)$\\
$1+x^{11}+x^{14}$ &$5115$ &$(2600,2515)$\\
&&\\
$1+x^9+x^{19}$ &$174251$ &$(87136,87115)$\\
$1+x^{10}+x^{19}$ &$174251$ &$(87136,87115)$\\
&&\\
\end{tabular}
\caption{All robust trinomials of degree less than or equal to $19$}\label{robust trinomials table}
\end{center}
\end{table}

Of all the polynomials studied in these various methods, the most interesting ones remain the families described in this section, due to the large ratio of the first coordinate of $\beta(f(x))$ to the second coordinate and because those were the only cases in which we were able to take the specific examples we noticed in the data and generalize to entire families of robust polynomials.  Samples of the code used in these search methods, as well as tables containing information on all robust polynomials of order less than or equal to $83$ and a complete list of all robust quadrinomials of degree less than or equal to $18$, can be seen in the appendices of my dissertation, available online at website to be inserted.

\section{Applications to Generalized Binary Representations}\label{Applications}

Every non-negative integer $n$ has a unique standard binary representation and can be written as a sum of powers of $2$ in the form
\[
 n=\sum_{i=0}^{\infty}\epsilon_i2^i,\quad \epsilon_i\in\{0,1\}.
\]
If we let $f_{\{0,1\}}(n)$ denote the number of ways to write $n$ in this fashion, then $f_{\{0,1\}}(n)=1$ for all $n\geq0$, as shown by Euler \cite[pages ~277--8]{Euler standard binary rep}.

Now consider instead the coefficient set $\{0,1,2\}$ and let $f_{\{0,1,2\}}(n)$ denote the number of ways to write $n$ as
\[
  n=\sum_{i=0}^{\infty}\epsilon_i2^i,\quad \epsilon_i\in\{0,1,2\}.
\]
First note that while it is still possible to represent every non-negative integer in this fashion, the representation is no longer unique.  For example, there are three ways to write $n=4$ as $\sum\epsilon_i2^i$ with $\epsilon_i\in\{0,1,2\}$, and they are
\[
 4=2\cdot1+1\cdot2=0\cdot1+0\cdot2+1\cdot2^2=0\cdot1+2\cdot2.
\]
Reznick showed in \cite{Reznick Binary Partition Functions} that when taking coefficients from the set $\{0,1,2\}$, the number of representations of $n-1$ corresponds to the $n^{th}$ term of the Stern sequence, which is defined recursively by $s(2n)=s(n)$ and $s(2n+1)=s(n)+s(n+1)$ with initial values $s(0)=0$ and $s(1)=1$.  The Stern sequence can also be viewed as a diatomic array in which each row is formed by inserting the sum of consecutive terms between the terms of the previous row.  This diatomic array is symmetric and is like a Pascal's triangle with memory.  The first few rows of this infinite array are shown in Table 2.

\begin{table}[h]
\begin{center}
\renewcommand{\arraystretch}{3}
\resizebox{12.5cm}{!}{
\begin{tabular}{ccccccccccccccccccccccccccccccccc}
 & & & & & & & & & & & & & & & 1 &  & 1 & & & & & & & &  \\
 & & & & & & & & & & & & & & & 1 & 2 & 1 & & & & & & & &  \\
 & & & & & & & & & & & & & & 1 & 3 & 2 & 3 & 1 & & & & & & &  \\
 & & & & & & & & & & & & 1 & 4 & 3 & 5 & 2 & 5 & 3 & 4 & 1 & & & & &  \\
 & & & & & & & & 1 & 5 & 4 & 7 & 3 & 8 & 5 & 7 & 2 & 7 & 5 & 8 & 3 & 7 & 4 & 5 & 1 &  \\
 1 & 6 & 5 & 9 & 4 & 11 & 7 & 10 & 3 & 11 & 8 & 13 & 5 & 12 & 7 & 9 & 2 & 9 & 7 & 12 & 5 & 13 & 8 & 11 & 3 & 10 & 7 & 11 & 4 & 9 & 5 & 6 & 1\\
\ldots  & 14 & 11 & 19 & 8 & 21 & 13 & 18 & 5 & 17 & 12 & 19 & 7 & 16 & 9 & 11 & 2 & 11 & 9 & 16 & 7 & 19 & 12 & 17 & 5 & 18 & 13 & 21 & 8 & 19 & 11 & 14 & \ldots\\
 \end{tabular}
}
 \end{center}\label{Stern array}\caption{Stern diatomic array}
\end{table}

To generalize these ideas, let $\mathcal{A}=\{0=a_0<a_1<\cdots<a_j\}$ denote a finite subset of $\mathbb{N}$ containing $0$.  We must include $0$ to avoid summing infinitely many powers of $2$.  Let $f_\mathcal{A}(n)$ denote the number of ways to write $n$ in the form
\[
 n=\sum_{i=0}^{\infty}\epsilon_i2^i,\quad \epsilon_i\in\mathcal{A}.
\]

We associate to $\mathcal{A}$ its characteristic function $\chi_\mathcal{A}(n)$. The generating function for $\chi_\mathcal{A}(n)$ is
\[
 \phi_\mathcal{A}(x):=\sum_{n=0}^{\infty}\chi_\mathcal{A}(n)x^n=\sum_{a\in\mathcal{A}}x^a=1+x^{a_1}+\cdots+x^{a_j}.
\]
Since $\mathcal{A}$ is a finite set, $\phi_\mathcal{A}(x)$ is a polynomial in $\mathbb{F}_2[x]$.  For example, we return to the specific cases discussed earlier and see that $\phi_{\{0,1\}}(x)=1+x$ and $\phi_{\{0,1,2\}}(x)=1+x+x^2$.

Denote the generating function of $f_\mathcal{A}(n)$ by
\[
F_\mathcal{A}(x):=\sum_{n=0}^{\infty}f_\mathcal{A}(n)x^n.
\]
Then 
\[
F_{\{0,1\}}(x)=\sum_{n=0}^{\infty}x^n=1+x+x^2+\cdots.
\]
If we view the number of ways to write $n$ as a partition problem, we obtain a product representation for $F_\mathcal{A}(x)$ as
\begin{equation}\label{product representation for F}
F_{\mathcal{A}}(x)=\prod_{k=0}^\infty\left(1+x^{a_12^k}+\cdots+x^{a_j2^k}\right)=\prod_{k=0}^\infty \phi_{\mathcal{A}}(x^{2^k}).
\end{equation}

In \cite{4 paper}, Anders, Dennison, Lansing, and Reznick studied the behavior of the sequence $(f_{\mathcal{A}}(n))\bmod 2$ and proved the following theorem.  

\begin{theorem}[{\cite[1.1]{4 paper}}]\label{Our Main Theorem}
As elements of the formal power series ring $\mathbb{F}_2[[x]]$,
\[
 \phi_\mathcal{A}(x)F_\mathcal{A}(x)=1.
\]
Hence $F_{\mathcal{A}}(x)\in\mathbb{F}_2(x)$.
\end{theorem}

Returning to the coefficient set $\{0,1,2\}$ with $\phi_{\{0,1,2\}}(x)=1+x+x^2$, we see by Theorem \ref{Our Main Theorem} that in $\mathbb{F}_2[[x]]$,
\begin{align*}
F_{\{0,1,2\}}(x)&=\frac{1}{1+x+x^2}\\
&=\frac{1+x}{1+x^3}\\
&=(1+x)(1+x^3+x^6+\cdots)\\
&=1+x+x^3+x^4+x^6+x^7+\cdots.
\end{align*}

Since $\mathcal{A}$ is finite, $\phi_{\mathcal{A}}(x)$ is a polynomial in $\mathbb{F}_2[x]$.  Also recall that the \textit{order} of $\phi_\mathcal{A}$ is the smallest integer $D$ such that $\phi_{\mathcal{A}}(x)\mid 1+x^D$.  Define $\phi_{\mathcal{A}}^*(x)$ by
\[
 \phi_{\mathcal{A}}(x)\phi_{\mathcal{A}}^*(x)=1+x^D.
\]
In coding theory, if $\operatorname{deg}\left( \phi_{\mathcal{A}}\right)=d$ and $D=2^d-1$, $\phi_\mathcal{A}(x)$ is called the \textit{generator polynomial}, while $\phi_\mathcal{A}^*(x)$ is called the \textit{parity-check polynomial}, \cite[page 484]{Lidl&N}.  We do not pursue these here.

Now we have in $\mathbb{F}_2[x]$,
\[
 F_{\mathcal{A}}(x)=\frac{1}{\phi_{\mathcal{A}}(x)}=\frac{\phi_{\mathcal{A}}^*(x)}{1+x^D}.
\]
If $\phi_{\mathcal{A}}^*(x)=\sum_{i=0}^r x^{b_i}$, where $0=b_0<b_1<\dots<b_r=D-\operatorname{max}\{a_i\}$, then
\[
 f_{\mathcal{A}}(n)\equiv 1\bmod 2 \Longleftrightarrow n\equiv b_i\bmod D \text{ for some } i.
\]

Cooper, Eichhorn, and O'Bryant considered the fraction $\gamma(f(x))$ defined in, as did we in \cite{4 paper}, but here we have instead considered the ordered pair
\begin{equation}\label{beta definition}
\beta(f(x))=(\ell_1(f^*), \ell_{0,D-1}(f^*)),
\end{equation}
which gives more precise information than reduced fractions.  In this pair, the first coordinate represents the number of times $f_\mathcal{A}(n)$ is odd in a minimal period $D$, and the second coordinate represents the number of times $f_\mathcal{A}(n)$ is even in a minimal period.

In light of the definitions presented in this section, we can restate Theorems \ref{maintheorem} and \ref{Second Main Theorem}, respectively, as follows.

\begin{theorem*}
Let $\mathcal{A}_r=\{0,1,2^r-1, 2^r+1\}$.  Then $\phi_{\mathcal{A}_r}(x)=f_{r,1}(x)$ and the sequence $\left(f_{\mathcal{A}_r}(n)\right) \bmod 2$ is periodic with least period dividing $4^r-1$.  Among $4^r-1$ consecutive terms of $\left(f_{\mathcal{A}_r}(n)\right)$, $4^r-3^r$ terms are odd and $3^r-1$ terms are even.
\end{theorem*}

\begin{theorem*}
Let $\mathcal{A}_r=\{0,1,2^r, 2^r+2\}$.  Then $\phi_{\mathcal{A}_r}(x)=f_{r,2}(x)$ and the sequence $\left(f_{\mathcal{A}_r}(n)\right) \bmod 2$ is periodic with least period dividing $4^r+2^r+1$.  Among $4^r+2^r+1$ consecutive terms of $\left(f_{\mathcal{A}_r}(n)\right)$, $4^r-3^r+2^r$ terms are odd and $3^r+1$ terms are even.
\end{theorem*}


\section{Open Questions}\label{Open Questions}
In this section, we discuss open questions relating to the problems, theorems, and examples in Section \ref{Families of Robust Polynomials}.

The original statement by Cooper, Eichhorn, and O'Bryant in \cite{KOB paper} was, ``The most interesting issued raised in this section, which remains unanswered, is to describe the set $\{\gamma\left(P\right): P \text{ is a polynomial}\}$.  For example, is there an $n$ with $\gamma\left(P_n\right)=3/4$?''  It is trivial that the infimum of the set is $0$, and we saw in Section \ref{Families of Robust Polynomials} that the supremum of the set is $1$.  The cluster points of the set remain to be determined, as does whether or not $3/4$ belongs to the set.

We would like to show that $4^r-1$ is in fact the order of the robust polynomials $f_{r,1}$ of Theorem \ref{maintheorem} and their reciprocals $f_{(R),r,1}$ of Corollary \ref{maintheorem corollary} rather than a multiple of the order, which is the result we now have.  Similarly, we hope to show that $4^r+2^r+1$ is the exact order of the robust polynomials $f_{r,2}$ and $f_{(R),r,2}$ of Theorem \ref{Second Main Theorem} and Corollary \ref{Second Main Theorem Corollary}.  

We would also like to find more families of robust polynomials.  It seems that the best way to do this would be to proceed as before, collecting large amounts of data and working to generalize the specific robust polynomials found in that data.  More efficient computing and coding will be needed, however, to obtain more data.

Another open problem is to consider properties of $f_\mathcal{A}(n)$ in bases other than $2$.  Calculations of sequences $\left(f_\mathcal{A}(n)\right)\bmod 3$ for $\mathcal{A}=\{0,1,3\},\{0,2,3\},\{0,1,4,9\}$, $\{0,1,5,9,10\},\{0,2,3,4\}, \{0,1,2,\dots,2^j\}$ for $2\leq j\leq 6$, and $\{0,1,3,\ldots,3^j\}$ for $2\leq j\leq 4$ showed no immediately obvious periodicity properties.  We also considered the sequence $\left(f_{\{0,2,8,9\}}(n)\right)$ mod $3,4,5,6,7\text{, and }8$ but noticed no periodicities.

If $f(x)=x^k+a_{k-1}x^{k-1}+a_{k-2}x^{k-2}+\cdots+a_1x+a_0$ with all $a_i\in\mathbb{F}_2$ and $a_0=1$ and $\mathcal{A}=\{j: 1\leq j\leq k\text{ and }a_j\neq 0\}$ , Theorem 8.78 of \cite{Lidl&N} gives 
\begin{align*}
&\left|\text{difference between \#0's and \#1's in a cycle of length $\operatorname{ord}(f(x))$ of }f_\mathcal{A}(n)\right|\\
&\quad =\left|\ell_1\left(f^*(x)\right)-\ell_{0,M-1}\left(f^*(x)\right)\right|\\
&\quad \leq 2^{k/2}.
\end{align*}

Hence $2^{k/2}$ is an upper bound for the difference in the coordinates of $\beta(f(x))$ for any $f(x)\in\mathbb{F}_2[x]$ of degree $k$ with constant term $1$.

The most extreme robust examples found thus far are those given in Examples \ref{f_3,1}, \ref{f_(R),3,1}, \ref{f_3,2}, and \ref{f_(R),3,2}.  Recall $\beta\left(f_{3,1}(x)\right)=\beta\left(f_{(R),3,1}(x)\right)=(37,26)$ and $\operatorname{deg}\left(f_{3,1}(x)\right)=\operatorname{deg}\left(f_{(R),3,1}(x)\right)=9$.  Additionally, $\beta\left(f_{3,2}(x)\right)=\beta\left(f_{(R),3,2}(x)\right)=(45,28)$ and $\operatorname{deg}\left(f_{3,2}(x)\right)=\operatorname{deg}\left(f_{(R),3,2}(x)\right)=10$.  Using the bound from \cite{Lidl&N}, we have
\begin{align*}
 \left|\ell_1\left(f_{3,1}^*(x)\right)-\ell_{0,62}\left(f_{3,1}^*(x)\right)\right|&=\left|\ell_1\left(f_{(R),3,1}^*(x)\right)-\ell_{0,62}\left(f_{(R),3,1}^*(x)\right)\right|\\
  &=37-26=11\leq 2^{9/2}\approx 22.6
\end{align*}
and
\begin{align*}
 \left|\ell_1\left(f_{3,2}^*(x)\right)-\ell_{0,72}\left(f_{3,2}^*(x)\right)\right|&=\left|\ell_1\left(f_{(R),3,2}^*(x)\right)-\ell_{0,72}\left(f_{(R),3,2}^*(x)\right)\right|\\
  &=45-28=17\leq 2^{10/2}=32.
\end{align*}

For any $r\geq 3$, if we assume that $4^r-1$ and $4^r+2^r+1$ are the exact orders of $f_{r,1}$ and $f_{r,2}$, respectively, we have
\begin{align*}
\left|\ell_1\left(f_{r,1}^*(x)\right)-\ell_{0,4^r-2}\left(f_{r,1}^*(x)\right)\right|&=\left|\ell_1\left(f_{(R),r,1}^*(x)\right)-\ell_{0,4^r-2}\left(f_{(R),r,1}^*(x)\right)\right|\\
&=4^r-3^r-(3^r-1)\\
&=4^r-2\cdot3^r+1\\
&\ll 2^{\frac{1}{2}(2^r+1)}\\
&=4^{2^{r-2}+\frac{1}{4}}
\end{align*}
and
\begin{align*}
\left|\ell_1\left(f_{r,2}^*(x)\right)-\ell_{0,4^r+2^r}\left(f_{r,2}^*(x)\right)\right|&=\left|\ell_1\left(f_{(R),r,2}^*(x)\right)-\ell_{0,4^r+2^r}\left(f_{(R),r,2}^*(x)\right)\right|\\
&=4^r-3^r+2^r-(3^r+1)\\
&=4^r-2\cdot3^r+2^r-1\\
&\ll 2^{\frac{1}{2}(2^r+2)}\\
&=4^{2^{r-2}+\frac{1}{2}},
\end{align*}
where the penultimate expressions in both displayed equations come from the upper bound in \cite{Lidl&N}.

Since these are the most extreme examples but do not push the upper bound, we suspect that the bound of $2^{k/2}$ could be improved in the $\mathbb{F}_2$ case.

\end{document}